\documentclass[a4paper,10pt]{amsart}

\usepackage{amsthm,amsmath,amssymb,amsfonts}
\usepackage[utf8]{inputenc}
\usepackage{verbatim}
\usepackage{marvosym}
\usepackage{stackrel}
\usepackage{graphicx}
\usepackage[svgnames]{xcolor}
\definecolor{blue(munsell)}{rgb}{0.0, 0.5, 0.69}
\usepackage{lipsum}
\usepackage{svg}
\usepackage{caption}
\usepackage{enumitem}
\usepackage{epigraph}
\usepackage{ebproof}
\usepackage[cal=boondoxo]{mathalfa}
\usepackage[all,cmtip]{xy}
\usepackage{mathtools}
\usepackage{color}
\usepackage{mathrsfs}
\usepackage{subfig}
\usepackage{framed}
\usepackage{hyperref}
\hypersetup{
    colorlinks = true,
    linkbordercolor = {red},
    linkcolor ={teal},
    anchorcolor = {pink},
    citecolor =  {orange},
    filecolor = {teal},
    menucolor = {teal},
    runcolor =  {teal},
    urlcolor = {teal},
}
\usepackage{cleveref}
\usepackage{tikz}
\usepackage{tikz-cd}
\usepackage{bbold}

\usepackage[colorinlistoftodos]{todonotes}
\usepackage{xparse}
\DeclareDocumentCommand\issue{g}{\todo[size=\footnotesize,color = green!40]{Issue\IfNoValueF{#1}{: #1}}}
\DeclareDocumentCommand\tobedone{g}{\todo[size=\footnotesize,color = yellow!50]{To do\IfNoValueF{#1}{: #1}}}
\DeclareDocumentCommand\notationissue{g}{\todo[size=\footnotesize,color = red!30]{Notation?\IfNoValueF{#1}{: #1}}}
\DeclareDocumentCommand\doubt{g}{\todo[size=\footnotesize,color = blue!10]{Doubt\IfNoValueF{#1}{: #1}}}
\DeclareDocumentCommand\observation{g}{\todo[size=\footnotesize,color = orange!10]{Observation\IfNoValueF{#1}{: #1}}}

\usepackage[all,cmtip]{xy}

\usetikzlibrary{calc}
\usetikzlibrary{matrix,arrows,arrows.meta}
\usepackage{tikz-cd}

\usepackage{calligra} 

\newtagform{fn}{(}{)\footnotemark}

    {\endMakeFramed}

    {\endMakeFramed}

\makeatletter
\g@addto@macro\bfseries{\boldmath}
\makeatother

\theoremstyle{definition}
\newtheorem{thm}{Theorem}[section]
\newtheorem*{thm*}{Theorem}
\newtheorem*{structure*}{Structure}
\newtheorem{prop}[thm]{Proposition}
\newtheorem{lem}[thm]{Lemma}

\newtheorem{con}[thm]{Construction}
\newtheorem{defn}[thm]{Definition}
\newtheorem{rem}[thm]{Remark}

\newtheorem{exa}[thm]{Example}

\newcommand\lan{\mathsf{lan}}
\newcommand\ran{\mathsf{ran}}

\newcommand\Set{\operatorname{\bf Set}}

\newcommand\op{\circ}

\newcommand\ca{\mathcal {A}}
\newcommand\cb{\mathcal {B}}
\newcommand\cc{\mathcal {C}}

\newcommand\cf{\mathcal {F}}

\newcommand\cg{\mathcal {G}}

\newcommand\ce{\mathcal {E}}

\newcommand\ck{\mathcal {K}}

\newcommand\cp{\mathcal {P}}

\newcommand\cy{\mathcal {Y}}

\DeclareFontFamily{U}{min}{}
\DeclareFontShape{U}{min}{m}{n}{<-> udmj30}{}
\newcommand\yo{\!\text{\usefont{U}{min}{m}{n}\symbol{'210}}\!}

\title{General facts on the Scott adjunction}
\author{Ivan Di Liberti$^\dag$}

\thanks{$^\dag$ This research was mostly developed during the PhD studies of the author and has been supported through the grant 19-00902S from the Grant Agency of the Czech Republic. The finalization of this research has been supported by the GACR project EXPRO 20-31529X and RVO: 67985840.}

\address{
\newline Ivan \textsc{Di Liberti}\newline
Institute of Mathematics\newline
Czech Academy of Sciences\newline
\v{Z}itn\'{a} 25, Prague, Czech Republic\newline
diliberti.math@gmail.com\newline
}

\begin{document}

\begin{abstract}
We introduce, comment and develop the Scott adjunction, mostly from the point of view of a category theorist. Besides its technical and conceptual aspects, in a nutshell we provide a categorification of the Scott topology over a posets with directed suprema. From a technical point of view we establish an adjunction between accessible categories with directed colimits and Grothendieck topoi. We show that the bicategory of topoi is enriched over the $2$-category of accessible categories with directed colimits and it has tensors with respect to this enrichment. The Scott adjunction (re-)emerges naturally from this observation.

\end{abstract}
\maketitle
\setcounter{tocdepth}{1}
{
  \hypersetup{linkcolor=black}
  \tableofcontents
}

\section{Introduction}\label{promenade}

In the 1980's Scott's work on dcpos has deeply impacted domain theory \cite{abramsky1994domain}, with motivations coming from theoretical computer sciences. Among his contributions, he introduced the \textit{Scott topology} on a poset with directed colimits. Given such a poset, $(P, \leq)$ we can define a topology on $P$ where opens are precisely those subsets whose characteristic function preserves directed joins. This construction amounts to a functor from the category of posets with directed colimits to the category of locales, \[\mathsf{S}: \text{Pos}_\omega \to \text{Loc} \] which assigns to each poset with directed colimits, its frame of Scott opens. Since its introduction, the Scott topology on a poset has been studied in deep detail. Yet, the study of the functorial properties of this construction did not have much luck. This becomes even more evident when we notice that the Scott construction happens to be the left adjoint of the functor of points $\mathsf{pt}: \text{Loc} \to \text{Pos}_\omega$, \[\mathsf{S} : \text{Pos}_\omega \leftrightarrows  \text{Loc}: \mathsf{pt}.\]
Indeed the crude set of points of a locale admits a structure of poset with directed colimits \cite{Vickers2007}, and the Scott construction offers a left adjoint for this assignement. Surprisingly, such observation has never appeared in the literature. Our task is to categorify the Scott construction. From a technical point of view we establish a biadjunction between accessible categories with directed colimits and Grothendieck topoi, \[\mathsf{S} : \text{Acc}_\omega \leftrightarrows  \text{Topoi}_\omega: \mathsf{pt}.\] We study its main properties and try to unveil its ultimate (categorical) nature.

Our analysis leads to the main result that the bicategory of topoi is enriched over accessible categories with directed colimits, and we rediscover the Scott adjunction as an instance of being tensored. 

\begin{thm*}[{\Cref{accmonoidalclosed}, \Cref{enrichment} and \Cref{tensored}}]
The $2$-category $\text{Acc}_\omega$ is monoidal closed, having the expected internal homs. The bicategory $\text{Topoi}$ is enriched and tensored over $\text{Acc}_\omega$.
\end{thm*}

Together with \cite{thgeo}, \cite{thlo}, this is one of three papers which accounts on the content of the author's Ph.D thesis. This paper is devoted to settle the categorical framework in which the theory is developed. The qualitative content of the adjunction is twofold. On one hand it has a very clean geometric interpretation, whose roots belong to Stone-like dualities and Scott domains. On the other, it can be seen as a syntax-semantics duality between formal model theory and geometric logic. Those points of view will be separately inspected in \cite{thgeo} and \cite{thlo}. 

\subsection{Our contribution} This paper has a predecessor, namely a joint collaboration of the author together with Simon Henry \cite{simon}. There, Henry uses the technology offered by the Scott adjunction to solve an axiomatizabiliy question asked by Jiří Rosický at the international conference Category Theory 14.  Our \Cref{promenade} is refined and more friendly presentation of the results contained in \cite{simon}. The main contribution of this paper are \Cref{sec3} and \Cref{toolbox}. \Cref{sec3} puts the Scott adjunction into perspective, recovering the adjunction from the fact that Topoi are enriched over accessible categories. \Cref{toolbox} is intended to be a collection of relevant and technical results on the adjunction that are rather useful for the applications in \cite{thgeo} and \cite{thlo}.

\subsection{Structure}
\begin{enumerate}
	\item[Sec. 2] accounts on \textit{the Scott construction}, mapping an accessible category with directed colimits to its \textit{Scott topos} (\Cref{defnS}), $\ca \mapsto \mathsf{S}(\ca).$ This amounts to a functor which is  left adjoint to taking the category of points of a topos $\ce \mapsto \mathsf{pt}(\ce).$
	\item[Sec. 3] makes a step back and studies some $2$-categorical property of the $2$-category of topoi, showing that Topoi is enriched over the $2$-category of accessible categories with directed colimits and it has tensors with respect to this enrichment ({\Cref{accmonoidalclosed}, \Cref{enrichment} and \Cref{tensored}}). The Scott adjunction (re-)emerges naturally from this observation (\Cref{ssrevisited}). The section also discusses the existence of certain cotensors (\Cref{powerspowers}).
	\item[Sec. 4] accounts on a collection of technical properties of the functors $\mathsf{S}$ and $\mathsf{pt}$, such as the preservation of relevant notions of monomorphisms. We introduce and study the notion of \textit{topological embedding} (\Cref{topological}) of accessible categories with directed colimits. Finally, we discuss the behavior of the unit of the Scott adjunction, and connect it to the existence of \textit{finitary representation} (\Cref{ff}). 
	\end{enumerate}

\subsection*{Notations and conventions} \label{backgroundnotations}

Most of the notation will be introduced when needed and we will try to make it as natural and intuitive as possible, but we would like to settle some notation.
\begin{enumerate}
\item $\ca, \cb$ will always be accessible categories, very possibly with directed colimits.
\item $\mathsf{Ind}_\lambda$ is the free completion under $\lambda$-directed colimits.
\item $\ca_{\kappa}$ is the full subcategory of $\kappa$-presentable objects of $\ca$.
\item $\cg, \cf, \ce$ will be Grothendieck topoi.
\item In general, $C$ is used to indicate small categories.
\item $\eta$ is the unit of the Scott adjunction.
\item $\epsilon$ is the counit of the Scott adjunction.
\item A Scott topos is a topos of the form $\mathsf{S}(\ca)$.
\item  $\P(X)$ is the category of small copresheaves of $X$.
\item $\text{Acc}_{\omega}$ is the $2$-category of accessible categories with directed colimits, a $1$-cell is a functor preserving directed colimits, $2$-cells are natural transformations.
\item $\omega$-$\text{Acc}$ is the full ($2$)-subcategory of  $\text{Acc}_{\omega}$ spanned by \textit{finitely} accessible categories.
\item $\text{Topoi}$ is the $2$-category of Grothendieck topoi. A $1$-cell is a geometric morphism and has the direction of the right adjoint. $2$-cells are natural transformation between left adjoints.
\item $\text{Presh}$ is the full  ($2$)-subcategory of  $\text{Topoi}$ spanned by presheaf categories.
\end{enumerate}

\section{The Scott Adjunction}\label{promenade}

In this section we provide enough information to understand the crude statement of the adjunction and we touch on these contextualizations. One could say that this section, together with a couple of results that appear in the Toolbox, is a report of our collaboration with Simon Henry \cite{simon}.

\begin{rem}
 Our presentation enriches Herny's exposition of some useful insights. \Cref{f_*} provides a more explicit description of the functoriality of the Scott construction. \Cref{scottname}, \Cref{duality}, \Cref{Free geometric theory} are dedicated to put the Scott adjunction into a more precise context, both from a conceptual and an historical viewpoint. \Cref{fields} provides an evident example that the Scott adjunction is not a biequivalence of bicategories. \Cref{diagrams} reproduces a standard result of topos theory in our language. Finally \Cref{promenadeproofofscott} provides a more polished and elegant proof of the Scott adjunction. Overall, given the terse style of \cite{simon} we think that this section is the perfect introduction to this technique.
 \end{rem}

\subsection{The Scott adjunction: definitions and constructions} \label{promedadescottadj}


We begin by giving the crude statement of the adjunction, then we proceed to construct and describe all the objects involved in the theorem. The actual proof of \Cref{scottadj} will close the section.

\begin{thm}[{\cite[Prop. 2.3]{simon}} The Scott adjunction]\label{scottadj}
The $2$-functor of points $\mathsf{pt} :\text{Topoi} \to \text{Acc}_{\omega} $ has a left biadjoint $\mathsf{S}$, yielding the Scott biadjunction, $$\mathsf{S} : \text{Acc}_{\omega} \leftrightarrows \text{Topoi}: \mathsf{pt}. $$
\end{thm}

\begin{rem}[Characters on the stage] 
$\text{Acc}_{\omega}$ is the $2$-category of accessible categories with directed colimits, a $1$-cell is a functor preserving directed colimits, $2$-cells are natural transformations. $\text{Topoi}$ is the $2$-category of Grothendieck topoi. A $1$-cell is a geometric morphism and has the direction of the right adjoint. $2$-cells are natural transformation between left adjoints.
 \end{rem}
 
 \begin{rem}[$2$-categorical warnings]\label{ignorepseudo}
Both $\text{Acc}_{\omega}$ and Topoi are  $2$-categories, but most of the time our intuition and our treatment of them will be $1$-categorical, we will essentially downgrade the adjunction to a $1$-adjunction where everything works \textit{up to equivalence of categories}. We feel free to use any classical result about $1$-adjunction, paying the price of decorating any statement with the correct use of the word \textit{pseudo}. For example, right adjoints preserve pseudo-limits, and pseudo-monomorphisms.
 \end{rem}

\begin{rem}[The functor $\mathsf{pt}$]\label{pt}
The functor of points $\mathsf{pt}$ belongs to the literature since quite some time, $\mathsf{pt}$ is the covariant hom functor $\text{Topoi}(\Set, - )$. It maps a Grothendieck topos $\cg$ to its category of points, \[\cg \mapsto \text{Cocontlex}(\cg, \Set).\]
Of course given a geometric morphism $f: \cg \to \ce$, we get an induced morphism $\mathsf{pt}(f): \mathsf{pt}(\cg) \to \mathsf{pt}{(\ce)}$ mapping $p^* \mapsto p^* \circ f^*$. The fact that $\text{Topoi}(\Set, \cg)$ is an accessible category with directed colimits appears in the classical reference by Borceux as \cite[Cor. 4.3.2]{borceux_19943}, while the fact that $\mathsf{pt}(f)$ preserves directed colimits follows trivially from the definition.
\end{rem}

\subsubsection{The Scott construction}

 \begin{con}[The Scott construction]\label{defnS}
We recall the construction of $\mathsf{S}$ from \cite{simon}. Let $\ca$ be an accessible category with directed colimits.  $\mathsf{S}(\ca)$ is defined as the category the category of functors preserving directed colimits into sets. \[\mathsf{S}(\ca) = \text{Acc}_{\omega}(\ca, \Set).\]
For a functor $f: \ca \to \cb$ be a $1$-cell in $\text{Acc}_{\omega}$, the geometric morphism $\mathsf{S}f$ is defined by precomposition as described below.

\begin{center}
\begin{tikzcd}
\ca \arrow[dd, "f"] &  & \mathsf{S}\ca \arrow[dd, "f_*"{name=lower, description}, bend left=49] \\
                    &  &                                                \\
\cb                 &  & \mathsf{S}\cb \arrow[uu, "f^*"{name=upper, description}, bend left=49] 

 \ar[phantom, from=lower, to=upper, shorten >=1pc, "\dashv", shorten <=1pc]
\end{tikzcd}
\end{center}
$\mathsf{S}f = (f^* \dashv f_*)$ is defined as follows: $f^*$ is the precomposition functor $f^*(g) = g \circ f$. This is well defined because $f$ preserve directed colimits. $f^*$ is a functor preserving all colimits between locally presentable categories\footnote{This is shown in \Cref{scottconstructionwelldefined}.} and thus has a right adjoint by the adjoint functor theorem\footnote{Apply the dual of \cite[Thm. 3.3.4]{borceux_1994} in combination with \cite[Thm 1.58]{adamekrosicky94}.}, that we indicate with $f_*$. Observe that $f^*$ preserves finite limits because finite limits commute with directed colimits in $\Set$.
\end{con}

\begin{rem}[$\mathsf{S}(\ca)$ is a topos]\label{scottconstructionwelldefined}
 Together with \Cref{defnS} this shows that the Scott construction provides a $2$-functor $\mathsf{S}: \text{Acc}_{\omega} \to \text{Topoi}$. A proof has already appeared in \cite[2.2]{simon} with a practically identical idea. The proof relies on the fact that, since finite limits commute with directed colimits, the category $\mathsf{S}(\ca)$ inherits from its inclusion in the category of all functors $\ca \to \Set$ all the relevant exactness condition prescribed by Giraud axioms. The rest of the proof is devoted to provide a generator for $\mathsf{S}(\ca)$. In the proof below we pack in categorical technology the proof-line above. 
\end{rem}
\begin{proof}[Proof of \Cref{scottconstructionwelldefined}]
By definition $\ca$ must be $\lambda$-accessible for some $\lambda$. Obviously $\text{Acc}_\omega(\ca, \Set)$ sits inside $\lambda\text{-Acc}(\ca, \Set)$. Recall that $\lambda\text{-Acc}(\ca, \Set)$ is equivalent to $\Set^{\ca_\lambda}$ by the restriction-Kan extension paradigm and the universal property of $\mathsf{Ind}_\lambda$-completion.  This inclusion $i: \text{Acc}_\omega(\ca, \Set) \hookrightarrow \Set^{\ca_\lambda}$, preserves all colimits and finite limits, this is easy to show and depends on one hand on how colimits are computed in this category of functors, and on the other hand on the fact that in $\Set$ directed colimits commute with finite limits. By the adjoint functor theorem, $\text{Acc}_\omega(\ca, \Set)$ amounts to a coreflective subcategory of a topos whose associated comonad is left exact. By \cite[V.8 Thm.4]{sheavesingeometry}, it is a topos.
\end{proof}

\begin{rem}[A description of $f_*$]\label{f_*}
In order to have a better understanding of the right adjoint $f_*$, which in the previous remark was shown to exist via a special version of the adjoint functor theorem, we shall fit the adjunction $(f^* \dashv f_*)$ into a more general picture. We start by introducing the diagram below.

\begin{center}
\begin{tikzcd}
\mathsf{S}\ca \arrow[ddd, "\iota_\ca"] \arrow[rr, "f_*" description, bend right] &  & \mathsf{S}\cb \arrow[ddd, "\iota_\cb"] \arrow[ll, "f^*"'] \\
 &  &  \\
 &  &  \\
{\P(\ca)} \arrow[rr, "\ran_f" description, bend right] \arrow[rr, "\lan_f" description, bend left] &  & {\P(\cb)} \arrow[ll, "f^*" description]
\end{tikzcd}
\end{center}

\begin{enumerate}
\item By $\P(\ca)$ we mean the category of small copresheaves over $\ca$. Observe that the natural inclusion $\iota_\ca$ of $\mathsf{S}\ca$ in $\P(\ca)$ has a right adjoint\footnote{This will be shown in the remark below.} $r_{\ca}$, namely $\mathsf{S}\ca$ is coreflective and it coincides with the algebras for the comonad $\iota_\ca \circ r_\ca$. If we ignore the evident size issue for which $\P(\ca)$ is not a topos, the adjunction $\iota_\ca \dashv r_{\ca}$ amounts to a geometric surjection $r: \P(\ca) \to \mathsf{S}\ca$.

 \item The left adjoint $\lan_f$ to $f^*$ does exist because $f$ preserve directed colimits, while in principle $\ran_f$ may not exists because it is not possible to cut down the size of the limit in the ran-limit-formula. Yet, for those functors for which it is defined, it provides a right adjoint for $f^*$. Observe that since the $f^*$ on the top is the restriction of the $f^*$ on the bottom, and $\iota_{\ca, \cb}$ are fully faithful, $f_*$ has to match with $r_\cb \circ \ran_f \circ \iota_\ca$, when this composition is well defined, \[f_* \cong r_\cb \circ \ran_f \circ \iota_\ca,\]indeed the left adjoint $f^*$ on the top coincides with $f^* \circ \iota_\cb$ and by uniqueness of the right adjoint one obtains precisely the equation above. Later in the text this formula will prove to be useful. We can already use it to have some intuition on the behavior  $f_*$, indeed $f_*(p)$ is the best approximation of $\ran_f(p)$ preserving directed colimits. In particular if it happens for some reason that $\ran_f(p)$ preserves directed colimits, then this is precisely the value of $f_*(p)$.
\end{enumerate}
\end{rem}

\begin{rem}[$\mathsf{S}(\ca)$ is coreflective in $\P(\ca)$] \label{coreflective} We would have liked to have a one-line-motivation of the fact that the inclusion $i_\ca: \mathsf{S}(\ca) \to \P(\ca)$ has a right adjoint $r_\ca$, unfortunately this result is true for a rather technical argument. By a general result of Kelly, $i_\ca$ has a right adjoint if and only if $\lan_{i_\ca}(1_{\mathsf{S}(\ca)})$ exists and $i_\ca$ preserves it. Since $\mathsf{S}(\ca)$ is small cocomplete, if $\lan_{i_\ca}(1_{\mathsf{S}(\ca)})$ exists, it must be pointwise and thus $i$ will preserve it because it is a cocontinuous functor. Thus it is enough to prove that $\lan_{i_\ca}(1_{\mathsf{S}(\ca)})$ exists. Anyone would be tempted to apply \cite[3.7.2]{borceux_1994}, unfortunately $\mathsf{S}(\ca)$ is not a small category. In order to cut down this size issue, we use the fact that $\mathsf{S}(\ca)$ is a topos and thus have a dense generator $j: G \to \mathsf{S}(\ca)$. Now, we observe that \[\lan_{i_\ca}(1_{\mathsf{S}(\ca)}) = \lan_{i_\ca}(\lan_{j}(j))= \lan_{i_\ca \circ j} j.\] Finally, on the latter left Kan extension we can apply \cite[3.7.2]{borceux_1994}, because $G$ is a small category.
\end{rem}

\begin{rem}\label{generatorscotttopos}
Let $\ca$ be a $\lambda$-accessible category, then $\mathsf{S}(\ca)$ can be described as the full subcategory of $\Set^{\ca_\lambda}$ of those functors preserving $\lambda$-small $\omega$-filtered colimits. A proof of this observation can be found in \cite[2.2]{simon}, and in fact shows that $\mathsf{S}(\ca)$ has a generator.
\end{rem}

\subsection{Comments and suggestions} \label{promenadecommentsandsuggestions}

 \begin{rem}[Cameos in the existing literature]\label{scottname}
Despite the name, neither the adjunction nor the construction is due to Scott and was presented for the first time in \cite{simon}. It implicitly appears in special cases both in the classical literature \cite{elephant2} and in some recent developments \cite{anel}. Karazeris introduces the notion of Scott topos of a finitely accessible category $\ck$ in \cite{Karazeris2001}, this notion coincides with $\mathsf{S}(\ck)$, as the name suggests. In \cite{thgeo} we will make the connection with some seminal works of Scott and clarify the reason for which this is the correct categorification of a construction which is due to him. As observed in \cite[2.4]{simon}, the Scott construction is the categorification of the usual Scott topology on a poset with directed joins. This will help us to develop a geometric intuition on accessible categories with directed colimits; they will look like the underlying set of some topological space. We cannot say to be satisfied with this choice of name for the adjunction, but we did not manage to come up with a better name.
\end{rem}

\begin{rem}[The duality-pattern] \label{duality}
A duality-pattern is an adjunction that is contravariantly induced by a dualizing object. For example, the famous dual adjunction between frames and topological spaces \cite[Chap. IX]{sheavesingeometry},
\[  \mathcal{O} : \mathsf{Top}   \leftrightarrows  \mathsf{Frm}^\circ:  \mathsf{pt} \]
is induced by the Sierpinski space $\mathbb{T}$. Indeed, since it admits a natural structure of frame, and a natural structure of topological space the adjunction above can be recovered in the form \[  \mathsf{Top}(-,\mathbb{T}): \mathsf{Top} \leftrightarrows  \mathsf{Frm}^\circ   : \mathsf{Frm}(-,\mathbb{T}).\] Most of the known topological dualities are induced in this way. The interested reader might want to consult \cite{porst1991concrete}. Makkai has shown \cite{Makkai-Pitts,makkai88} that relevant families of syntax-semantics dualities can be induced in this way using the category of sets as a dualizing object.
In this fashion, the content of  \Cref{defnS} together with \Cref{pt} acknowledges that $\mathsf{S} \dashv \mathsf{pt}$ is essentially induced by the object $\Set$ that inhabits both the $2$-categories.
\end{rem}

\begin{rem}[Generalized axiomatizations] \label{Free geometric theory}
As was suggested by Joyal, the category $\text{Logoi} = \text{Topoi}^{\op}$ can be seen as the category of geometric theories. Caramello \cite{caramello2010unification} pushes on the same idea stressing the fact that a topos is a Morita-equivalence class of geometric theories. In this perspective the Scott adjunction, which in this case is a dual adjunction \[\text{Acc}_{\omega} \leftrightarrows \text{Logoi}^{\op}, \] presents $\mathsf{S}(\ca)$ as a free geometric theory attached to the accessible category $\ca$ that is willing to axiomatize $\ca$. When $\ca$ has a faithful functor preserving directed colimits into the category of sets, $\mathsf{S}(\ca)$ axiomatizes an envelope of $\ca$, as will be shown in \Cref{ff} and \Cref{fff}. A logical understanding of the adjunction will be developed in \cite{thlo}, where we connect the Scott adjunction to the theory of classifying topoi and to the seminal works of Lawvere and Linton in categorical logic. This intuition will be used also to give a topos theoretic approach to abstract elementary classes.
\end{rem}

\begin{rem}[Trivial behaviors and Diaconescu]\label{trivial}
 If $\ck$ is a finitely accessible category, $\mathsf{S}(\ck)$ coincides with the presheaf topos $\Set^{\ck_{\omega}}$, where we indicated with $\ck_{\omega}$ the full subcategory of finitely presentable objects. This follows directly from the following chain of equivalences, \[\mathsf{S}(\ck) = \text{Acc}_{\omega}(\ck, \Set)  \simeq  \text{Acc}_{\omega}(\mathsf{Ind}(\ck_{\omega}), \Set)  \simeq \Set^{\ck_{\omega}}. \] As a consequence of Diaconescu theorem \cite[B3.2.7]{elephant1} and the characterization of the $\mathsf{Ind}$-completion via flat functors, when restricted  to finitely accessible categories, the Scott adjunction yields a biequivalence of $2$-categories $\omega\text{-Acc} \simeq \text{Presh}$, with Presh the full $2$-subcategory of presheaf topoi. 
 \begin{center}
\begin{tikzcd}
\text{Acc}_{\omega} \arrow[rr, "\mathsf{S}" description, bend right] &  & \text{Topoi} \arrow[ll, "\mathsf{pt}" description, bend right] \\
 &  &  \\
\omega\text{-Acc} \arrow[uu, hook] \arrow[rr, bend right] &  & \text{Presh} \arrow[uu, hook] \arrow[ll, bend right]
\end{tikzcd}
\end{center}  
This observation is not new to literature, the proof of \cite[C4.3.6]{elephant2} gives this special case of the Scott adjunction. It is very surprising that the book does not investigate, or even mention the existence of the Scott adjunction, since it gets very close to defining it explicitly.
\end{rem}

\begin{thm}
The Scott adjunction restricts to a biequivalence of $2$-categories between the $2$-category of finitely accessible categories and the $2$-category of presheaf topoi.

$$\mathsf{S} : \omega\text{-Acc} \leftrightarrows \text{Presh}: \mathsf{pt}. $$
\end{thm}
\begin{proof}
The previous remark has shown that  when $\ca$ is finitely accessible, $\mathsf{S}(\ca)$ is a presheaf topos and that, when $\ce$ is a presheaf topos, $\mathsf{pt}(\ce)$ is finitely accessible. To finish, we show that in this case the unit and the counit of the Scott adjunction are equivalence of categories. This is in fact essentially shown by the previous considerations. $$(\mathsf{pt}\mathsf{S}) (\mathsf{Ind}(C)) \simeq \mathsf{pt}(\Set^C) \stackrel{\text{Diac}}{\simeq} \mathsf{Ind}(C). $$

$$(\mathsf{S}\mathsf{pt}) (\Set^C) \stackrel{\text{Diac}}{\simeq} \mathsf{S}(\mathsf{Ind}(C)) \simeq \Set^C. $$
Notice that the equivalences above are precisely the unit and the counit of the Scott adjunction as described in \Cref{promenadeproofofscott} of this section.
\end{proof}

\begin{rem} \label{trivial1}
Thus, the Scott adjunction must induce an equivalence of categories between the Cauchy completions\footnote{The free completions that adds splittings of \textit{pseudo}-idempotents.} of $\omega$-Acc and Presh. The Cauchy completion of $\omega$-Acc is the full subcategory of $\text{Acc}_\omega$ of \textit{continuous categories} \cite{cont}. Continuous categories are the categorification of the notion of continuous poset and can be characterized as split subobjects of finitely accessible categories in  $\text{Acc}_\omega$. In \cite[C4.3.6]{elephant2} Johnstone observe that if a continuous category is cocomplete, then the corresponding Scott topos is injective (with respect to geometric embeddings) and vice-versa.
\end{rem}

\begin{exa}
As a direct consequence of \Cref{trivial}, we can calculate the Scott topos of $\Set$. $\mathsf{S}(\Set)$ is  $\Set^{\text{FinSet}}$. This topos is very often indicated as $\Set[\mathbb{O}]$, being the classifying topos of the theory of objects, i.e. the functor: $\text{Topoi}(-, \Set[\mathbb{O}]): \text{Topoi}^{\circ} \to \text{Cat}$ coincides with the forgetful functor. As a reminder for the reader, we state clearly the equivalence: \[\mathsf{S}(\Set) \simeq \Set[\mathbb{O}]. \]
\end{exa}

\begin{rem}[The Scott adjunction is not a biequivalence: Fields] \label{fields}
Whether the Scott adjunction is a biequivalence is a very natural question to ask. Up to this point we noticed that on the subcategory of topoi of presheaf type the counit of the adjunction is in fact an equivalence of categories. Since presheaf topoi are a quite relevant class of topoi one might think that the whole bi-adjunction amounts to a biequivalence. That's not the case: in this remark we provide a topos $\cf$ such that the counit \[\epsilon_\cf : \mathsf{Spt} \cf \to \cf\] is not an equivalence of categories. Let $\cf$ be the classifying topos of the theory of geometric fields \cite[D3.1.11(b)]{elephant2}. Its category of points is the category of fields $\mathsf{Fld}$, since this category is finitely accessible the Scott topos $\mathsf{Spt}(\cf)$ is of presheaf type by \Cref{trivial}, \[\mathsf{Spt}(\cf) = \mathsf{S} (\mathsf{Fld} ) \stackrel{\ref{trivial}}{\cong} \Set^{\mathsf{Fld}_\omega}.\]
It was shown in \cite[Cor 2.2]{10.2307/30041767} that $\cf$ cannot be of presheaf type, and thus $\epsilon_\cf$ cannot be an equivalence of categories.

\end{rem}

\begin{rem}[Classifying topoi for categories of diagrams] \label{diagrams}
Let us give a proof in our framework of a well known fact in topos theory, namely that a the category of diagrams over the category of points of a topos can be axiomatized by a geometric theory. This means that there exists a topos $\cf$  such that 
$$\mathsf{pt}(\ce)^C \simeq  \mathsf{pt}(\cf).$$
The proof follows from the following chain of equivalences. 
\begin{align*} 
\mathsf{pt}(\ce)^C  = &   \text{Cat}(C,\mathsf{pt}(\ce))  \\ 
\simeq & \text{Acc}_\omega(\mathsf{Ind}(C), \mathsf{pt}(\ce)) \\
\simeq & \text{Topoi}(\mathsf{S}\mathsf{Ind}(C), \ce) \\
\simeq & \text{Topoi}(\Set^C, \ce) \\
\simeq & \text{Topoi}(\Set \times \Set^C, \ce) \\
\simeq & \text{Topoi}(\Set, \ce^{\Set^C}) \\
\simeq & \mathsf{pt}(\ce^{\Set^C}).
\end{align*}

Notice that $\Set^C$ is indeed an exponentiable topos because a presheaf topos is locally finitely presentable, and thus is a continuous category, so that we can apply the main theorem of \cite{cont}.
\end{rem}

\subsection{The $\kappa$-Scott adjunction}\label{kappa} \label{promenadegeneralizations}

The most natural generalization of the Scott adjunction is the one in which directed colimits are replaced with $\kappa$-filtered colimits and finite limits ($\omega$-small) are replaced with $\kappa$-small limits. This unveils the deepest reason for which the Scott adjunction exists: namely $\kappa$-directed colimits commute with $\kappa$-small limits in the category of sets.

\begin{thm}[{\cite[Prop 3.4]{simon}}]
There is  an $2$-adjunction $$\mathsf{S}_{\kappa} : \text{Acc}_{\kappa} \leftrightarrows \kappa\text{-Topoi}: \mathsf{pt}_{\kappa}. $$
\end{thm}

\begin{rem}
$\text{Acc}_{\kappa}$ is the $2$-category of accessible categories with $\kappa$-directed colimits, a $1$-cell is a functor preserving $\kappa$-filtered colimits, $2$-cells are natural transformations. $\text{Topoi}_{\kappa}$ is the $2$-category of Grothendieck $\kappa$-topoi. A $1$-cell is a $\kappa$-geometric morphism and has the direction of the right adjoint. $2$-cells are natural transformation between left adjoints. A $\kappa$-topos is a $\kappa$-exact localization of a presheaf category. These creatures are not completely new to the literature but they appear sporadically and a systematic study is still missing. We should reference, though, the works of Espindola \cite{ESPINDOLA2019137}. We briefly recall the relevant definitions.
 \end{rem}

 \begin{defn}
 A $\kappa$-topos is a $\kappa$-exact localization of a presheaf category.
 \end{defn}

  \begin{defn}
 A $\kappa$-geometric morphism $f: \ce \to \cf$ between $\kappa$-topoi is an adjunction $f^*: \cf \leftrightarrows \ce: f_*$ whose left adjoint preserves $\kappa$-small limits.
 \end{defn}
 
 \begin{rem}\label{generalize}
It is quite evident that every remark until this point  finds its direct $\kappa$-generalization substituting every occurrence of \textit{directed colimits }with \textit{$\kappa$-directed colimits.}
 \end{rem}

\begin{rem}\label{kdiac}
Let $\ca$ be a category in $\text{Acc}_{\omega}$.  For a large enough $\kappa$ the Scott adjunction axiomatizes $\ca$ (in the sense of \Cref{Free geometric theory}), in fact if $\ca$ is $\kappa$-accessible $\mathsf{pt}_\kappa \mathsf{S}_\kappa \ca \cong \ca$, for the $\kappa$-version of Diaconescu Theorem, that in this text appears in \Cref{trivial}.
\end{rem}

\begin{rem} It might be natural to conjecture that Presh happens to be $\bigcap_\kappa \kappa \text{-Topoi}$.  Simon Henry, has provided a counterexample to this superficial conjecture, namely $\mathsf{Sh}([0,1])$. \cite[Rem. 4.4, 4.5 and 4.6]{kelly1989complete} gives a theoretical reason for which many other counterexamples do exist and then provides a collection of them in Sec. 5 of the same paper.
\end{rem}

\subsection{Proof of the Scott Adjunction} \label{promenadeproofofscott}

We end this section including a full proof of the Scott adjunction.

\begin{proof}[Proof of \Cref{scottadj}] 
We prove that there exists an equivalence of categories, $$\text{Topoi}(\mathsf{S}(\ca), \cf) \cong \text{Acc}_\omega(\ca, \mathsf{pt}(\cf)). $$ The proof makes this equivalence evidently natural. This proof strategy is similar to that appearing in \cite{simon}, even thought it might look different at first sight.

\begin{align*} 
\text{Topoi}(\mathsf{S}(\ca), \cf) \cong &  \text{Cocontlex}(\cf, \mathsf{S}(\ca))  \\ 
 \cong & \text{Cocontlex}(\cf, \text{Acc}_\omega(\ca, \Set))  \\
\cong & \text{Cat}_{\text{cocontlex,}\text{acc}_\omega}(\cf \times \ca, \Set) \\
\cong &  \text{Acc}_\omega(\ca, \text{Cocontlex}(\cf,\Set))  \\
 \cong &  \text{Acc}_\omega(\ca, \text{Topoi}(\Set,\cf)). \\
  \cong & \text{Acc}_\omega(\ca, \mathsf{pt}(\cf)).
\end{align*}

\end{proof}

\begin{proof}[A description of the (co)unit]
We spell out the unit and the counit of the adjunction. 
\begin{itemize}
  \item[$\eta$] For an accessible category with directed colimits $\ca$ we must provide a functor $\eta_\ca: \ca \to \mathsf{pt}\mathsf{S}(\ca)$. Define, $$\eta_\ca(a)(-) := (-)(a).$$  $\eta_\ca(a): \mathsf{S}(\ca) \to \Set$ defined in this way  is a functor preserving colimits and finite limits and thus defines a point of $\mathsf{S}(\ca)$.
  \item[$\epsilon$] The idea is very similar, for a topos $\ce$, we must provide a geometric morphism $\epsilon_\ce: \mathsf{S}\mathsf{pt}(\ce) \to \ce$. Being a geometric morphism, it's equivalent to provide a cocontinuous and finite limits preserving functor $\epsilon_\ce^*: \ce \to \mathsf{S}\mathsf{pt}(\ce)$. Define, $$\epsilon_\ce^*(e)(-) = (-)^*(e). $$
\end{itemize}
\end{proof}

\section{Enrichment of Topoi over $\text{Acc}_\omega$} \label{sec3}

This section is dedicated to a $2$-categorical perspective on the Scott adjunction and its main characters. We provide an overview of the categorical properties of $\text{Acc}_\omega$ and Topoi. Mainly, we show that the $2$-category of topoi is enriched over $\text{Acc}_\omega$ and has copowers. We show that this observation generalizes the Scott adjunction in a precise sense. We discuss the $2$-categorical properties of both the $2$-categories, but this work is not original. We will provide references within the discussion.

\subsection{$2$-categorical properties of $\text{Acc}_\omega$}
\subsubsection{(co)Limits in $\text{Acc}_\omega$} \label{colimitsinacc}

The literature contains a variety of results on the $2$-dimensional structure of the $2$-category $\text{Acc}$ of accessible categories and accessible functors. Among these, one should mention \cite{Makkaipare} for lax and pseudo-limits in $\text{Acc}$ and \cite{colimitacc} for colimits. Our main object of study, namely $\text{Acc}_\omega$, is a (non-full) subcategory of $\text{Acc}$, and thus it is a bit tricky to infer its properties from the existing literature. Most of the work was successfully accomplished in \cite{lieberman2015limits}. Let us list the main results of these references that are related to $\text{Acc}_\omega$.

\begin{prop}[{\cite[2.2]{lieberman2015limits}}]
$\text{Acc}_\omega$ is closed under pie-limits\footnote{These are those limits can be reduced to products, inserters and equifiers.} in $\text{Acc}$ (and thus in the illegitimate $2$-category of locally small categories).
\end{prop}

\begin{prop}[Slight refinement of {\cite[2.1]{colimitacc}}]
Every directed diagram of accessible categories and full embeddings preserving directed colimits has colimit in Cat, and is in fact the colimit in $\text{Acc}_\omega$.
\end{prop}

\subsubsection{$\text{Acc}_\omega$ is monoidal closed}

This subsection discusses a monoidal closed structure on $\text{Acc}_\omega$. The reader should keep in mind the monoidal product of modules over a ring, because the construction is similar in spirit, at the end of the subsection we will provide an hopefully convincing argument in order to show that the construction is similar for a quite quantitative reason. The main result of the section should be seen as a slight variation of \cite[6.5]{kelly1982basic} where the enrichment base is obviously the category of Sets and $\mathcal{F}$-cocontinuity is replaces by preservation of directed colimits. Our result doesn't technically follows from Kelly's one because of size issues, but the general idea of the proof is in that spirit. Moreover, we found it clearer to provide an explicit construction of the tensor product in our specific case. The reader is encouraged to check \cite{350869}, where Brandenburg provides a concise presentations of Kelly's construction. For a treatment of how the bilinear tensor product on categories with certain colimits gives a monoidal bicategory we refer to \cite{bourke2017skew,HYLAND2002141,LOPEZFRANCO20112557}. 

\begin{rem}[A natural internal hom] \label{internalhom} Given two accessible categories $\ca, \cb $ in $\text{Acc}_\omega$, the category of functors preserving directed colimits $\text{Acc}_\omega(\ca, \cb)$ has directed colimits and they are computed pointwise. Moreover it is easy to show that it is sketchable and thus accessible. Indeed $\text{Acc}_\omega(\ca, \cb)$ is accessibly embedded in $\cb^{\ca_\lambda}$ and coincides with the category of those functors $\ca_\lambda \to \cb$ preserving $\lambda$-small directed colimits, which makes it clearly sketchable. Thus we obtain a $2$-functor, \[[-,-]:\text{Acc}_\omega^\circ \times \text{Acc}_\omega \to \text{Acc}_\omega.\] In our analogy, this corresponds to the fact that the set of morphisms between two modules over a ring $\mathsf{Mod}(M,N)$ has a (pointwise) structure of module.
\end{rem}

\begin{rem}[Looking for a tensor product: the universal property] \label{univproptensor} 
Assume for a moment that the tensorial structure that we are looking for exists, then we would obtain a family of (pseudo)natural equivalences of categories, $$\text{Acc}_\omega(\ca \otimes \cb, \cc) \simeq \text{Acc}_\omega (\ca, [\cb, \cc]) \simeq \omega\text{-Bicocont}(\ca \times \cb, \cc).$$ In the display we wrote $\omega\text{-Bicocont}(\ca \times \cb, \cc)$ to mean the category of those functors that preserve directed colimits in each variable. The equation gives us the universal property that should define $\ca \otimes \cb$ up to equivalence of categories and is consistent with our ongoing analogy of modules over a ring, indeed the tensor product classifies \textit{bilinear maps}. 
\end{rem}

\begin{con}[Looking for a tensor product: the construction] \label{constructiontensor} 
Let $\yo: \ca \times \cb \to \cp(\ca \times \cb)$ be the Yoneda embedding of $\ca \times \cb$ corestricted to the full subcategory of small presheaves \cite{presheaves}. Let $B(\ca,\cb)$ be the full subcategory of $\cp(\ca \times \cb)$ of those functors that preserve cofiltered limits in both variables\footnote{i.e., send filtered colimits in $\ca$ or $\cb$ to limits in $\Set$.}. It is easy to show that $B(\ca,\cb)$ is sketched by a limit theory, and thus is locally presentable. The inclusion $i: B(\ca,\cb) \hookrightarrow \cp(\ca \times \cb)$ defines a small-orthogonality class\footnote{Here we are using that $\ca$ and $\cb$ are accessible in order to cut down the size of the orthogonality.} in $\cp(\ca \times \cb)$ and is thus reflective \cite[1.37, 1.38]{adamekrosicky94}. Let $L$ be the left adjoint of the inclusion, as a result we obtain an adjunction, $$L: \cp(\ca \times \cb) \leftrightarrows B(\ca,\cb): i. $$
Now define $\ca \otimes \cb$ to be the smallest full subcategory of $B(\ca,\cb)$ closed under directed colimits and containing the image of $L \circ \yo$. Thm. \cite[6.23]{kelly1982basic} in Kelly ensures that $\ca \otimes \cb$ has the universal property described in \Cref{univproptensor} and thus is our tensor product. It might be a bit hard to see but this construction still follows our analogy, the tensor product of two modules is indeed built from free module on the product and  the \textit{bilinear} relations.
\end{con}

\begin{thm}\label{accmonoidalclosed}
$\text{Acc}_\omega$, together with the tensor product $\otimes$ defined in \Cref{constructiontensor} and the internal hom defined in \Cref{internalhom} is a monoidal biclosed bicategory\footnote{This is not strictly true, because the definition of monoidal closed category does not allow for equivalence of categories. We did not find a precise terminology in the literature and we felt non-useful to introduce a new concept for such a small discrepancy.} in the sense that there is are pseudo-equivalences of categories $$\text{Acc}_\omega(\ca \otimes \cb, \cc) \simeq \text{Acc}_\omega (\ca, [\cb, \cc]),$$ which are natural in $\cc$.
\end{thm}
\begin{proof}
Follows directly from the discussion above.
\end{proof}

\begin{rem}[Up to iso/up to equivalence]
As in \cite[6.5]{kelly1982basic}, we will not distinguish between the properties of this monoidal structure (where everything is true up to equivalence of categories) and a usual one, where everything is true up to isomorphism. In our study this distinction never plays a rôle, thus we will use the usual terminology about monoidal structures. 
\end{rem}

\begin{rem}[The unit] \label{unitmonoidalstructure}
The unit of the above-mentioned monoidal structure is the terminal category in Cat, which is also terminal in $\text{Acc}_\omega$.
\end{rem}

\begin{rem}[Looking for a tensor product: an abstract overview] \label{tensorabstract} 
In this subsection we have used the case of modules over a ring as a kind of analogy/motivating example. In this remark we shall convince the reader that the analogy can be pushed much further. Let's start by the observation that $R\text{-}\mathsf{Mod}$ is the category of algebras for the monad $R[-]: \Set \to \Set$. The monoidal closed structure of $\mathsf{Mod}$ can be recovered from the one of the category of sets $(\Set, 1, \times, [-,-])$ via a classical theorem proved by Kock in the seventies.  It would not make the tractation more readable to cite all the papers that are involved in this story, thus we mention the PhD thesis of Brandenburg \cite[Chap. 6]{brandenburg2014tensor} which provides a very coherent and elegant presentation of the literature. 

\begin{thm}(Seal, \cite[6.5.1]{brandenburg2014tensor}) Let $T$ be a coherent (symmetric) monoidal monad on a (symmetric) monoidal category C with reflexive coequalizers. Then $\mathsf{Mod}(T)$ becomes a (symmetric) monoidal category.
\end{thm}

Now similarly to $\mathsf{Mod}(R)$ the $2$-category of categories with directed colimits and functors preserving them is the category of (pseudo)algebras for the KZ monad of the Ind-completion over locally small categories $$\mathsf{Ind}: \text{Cat} \to \text{Cat}. $$ \cite[6.7]{bourke2017skew} provides a version of Seal's theorem for monads over Cat. While it's quite easy to show that the completion under directed colimits meets many of Bourke's hypotheses, we do not believe that it meets all of them, thus we did not manage to apply a Kock-like result to derive \Cref{accmonoidalclosed}. Yet, we think we have provided enough evidence that the analogy is not just motivational.
\end{rem}

\subsection{$2$-categorical properties of $\text{Topoi}$} \label{propoftopoi}
\subsubsection{(co)Limits in $\text{Topoi}$}

The $2$-categorical properties of the category of topoi have been studied in detail in the literature. We mention \cite[B3.4]{elephant2} and \cite{Lurie} as a main reference.

\subsubsection{Enrichment over $\text{Acc}_\omega$, tensor and powers} \label{topoienriched}

The main content of this sub-subsection will be to show that the category of topoi and geometric morphisms (in the direction of the right adjoint) is enriched over $\text{Acc}_\omega$. Notice that formally speaking, we are enriching over a monoidal bicategory, thus the usual theory of enrichment is not sufficient. As Garner pointed out to us, the theory of bicategories enriched in a monoidal bicategory is originally due to Bozapalides in the 1970s, though he was working without the appropriate technical notion; more precise definitions are in the PhD theses of Camordy and Lack; and everything is worked out in very deep detail in \cite{garner2016enriched}.

\begin{rem} Recall that to provide such an enrichment means to

\begin{enumerate}
    \item show that given two topoi $\ce, \cf$, the set of geometric morphisms $\text{Topoi}(\ce, \cf)$ admits a structure of accessible category with directed colimits.
    \item provide, for each triple of topoi $\ce, \cf, \cg$, a functor preserving directed colimits $$\circ: \text{Topoi}(\ce, \cf) \otimes \text{Topoi}(\cf, \cg) \to  \text{Topoi}(\ce, \cg), $$ making the relevant diagrams commute.
\end{enumerate} 
(1) will be shown in \Cref{topoihomacc}, while (2) will be shown in \Cref{topoicomposition}.
\end{rem}

\begin{prop}\label{topoihomacc}
 Let $\ce,\cf$ be two topoi. Then the category of geometric morphisms $\text{Cocontlex}(\ce, \cf)$, whose objects are cocontinuous left exact functors and morphisms are natural transformations is an accessible category with directed colimits.
\end{prop}

\begin{proof}
The proof goes as in \cite[Cor.4.3.2]{borceux_19943}, $\Set$ plays no rôle in the proof. What matters is that finite limits commute with directed colimits in a topos.
\end{proof}

\begin{prop}\label{topoicomposition}
For each triple of topoi $\ce, \cf, \cg$, there exists a functor preserving directed colimits $$\circ: \text{Topoi}(\ce, \cf) \otimes \text{Topoi}(\cf, \cg) \to  \text{Topoi}(\ce, \cg), $$ making the relevant diagrams commute.
\end{prop}
\begin{proof}
We will only provide the composition. The relevant diagrams commute trivially from the presentation of the composition. Recall that by \Cref{univproptensor} a map of the form $\circ: \text{Topoi}(\ce, \cf) \otimes \text{Topoi}(\cf, \cg) \to  \text{Topoi}(\ce, \cg), $ preserving directed colimits is the same of a functor $$ \circ: \text{Topoi}(\ce, \cf) \times \text{Topoi}(\cf, \cg) \to  \text{Topoi}(\ce, \cg) $$ preserving directed colimits in each variables. Obviously, since left adjoints can be composed in Cat, we already have such a composition. It's enough to show that it preserves directed colimits in each variable. Indeed this is the case, because directed colimits in these categories are computed pointwise.
\end{proof}

\begin{thm} \label{enrichment}
The category of topoi is enriched over $\text{Acc}_\omega$.
\end{thm}
\begin{proof}
Trivial from the previous discussion.
\end{proof}

\subsubsection{Tensors} \label{tensored}
In this subsection we show that the $2$-category of topoi has tensors (copowers) with respect to the enrichment of the previous section.

\begin{rem}
Let us recall what means to have tensors for a category $\mathsf{K}$ enriched over sets (that is, just a locally small category). To have tensors in this case means that we can define a functor $\boxtimes: \Set \times \mathsf{K} \to \mathsf{K}$ in such a way that, $$\mathsf{K}(S \boxtimes k, h) \cong \Set(S,\mathsf{K}( k, h)). $$ For example, the category of modules over a ring has tensors given by the formula $S \boxtimes M := \oplus_{S}M$; indeed it is straightforward to observe that $$R\text{-}\mathsf{Mod}(\oplus_S M, N) \cong \Set(S,R\text{-}\mathsf{Mod}(M, N)). $$ In this case, this follows from the universal property of the coproduct.
\end{rem}

\begin{rem}[The construction of tensors] \label{boxtimesdef}
We shall define a $2$-functor $\boxtimes: \text{Acc}_\omega \times \text{Topoi} \to \text{Topoi}.$ Our construction is reminiscent of the Scott adjunction, and we will see that there is an extremely tight connection between the two. Given a topos $\ce$ and an accessible category with directed colimits $\ca$ we define, $$\ca \boxtimes \ce := \text{Acc}_\omega(\ca, \ce).$$ In order to make this construction meaningful we need to accomplish two tasks:

\begin{enumerate}
    \item show that the construction is well defined (on the level of objects), that is, show that $\text{Acc}_\omega(\ca, \ce)$ is a topos.
    \item describe the action of $\boxtimes$ on functors.
\end{enumerate}

We split these two tasks into two different remarks.
\end{rem}

\begin{rem}[$\text{Acc}_\omega(\ca, \ce)$ is a topos] \label{tensorscottconstructionwelldefined}

By definition $\ca$ must be $\lambda$-accessible for some $\lambda$. Obviously $\text{Acc}_\omega(\ca, \ce)$ sits inside $\lambda\text{-Acc}(\ca, \ce)$. Recall that $\lambda\text{-Acc}(\ca, \ce)$ is equivalent to $\ce^{\ca_\lambda}$ by the restriction-Kan extension paradigm and the universal property of $\mathsf{Ind}_\lambda$-completion.  The inclusion $i: \text{Acc}_\omega(\ca, \ce) \hookrightarrow \ce^{\ca_\lambda}$, preserves all colimits and finite limits, this is easy to show and depends on the one hand on how colimits are computed in this category of functors, and on the other hand on the fact that in a topos directed colimits commute with finite limits. Thus $\text{Acc}_\omega(\ca, \ce)$ amounts to a coreflective subcategory of a topos whose associated comonad is left exact. So by \cite[V.8 Thm.4]{sheavesingeometry}, it is a topos.
\end{rem}

\begin{rem}[Action of $\boxtimes$ on functors] \label{tensor1cells}
Let $f: \ca \to \ce$ be in $\text{Acc}_\omega(\ca, \ce)$ and let $g:\ce \to \cf$ be a geometric morphism. We must define a geometric morphism $$f \boxtimes g: \text{Acc}_\omega(\ca, \ce) \to \text{Acc}_\omega(\cb, \cf).  $$
We shall describe the left adjoint $(f \boxtimes g)^*$ (which goes in the opposite direction ($f \boxtimes g)^*:  \text{Acc}_\omega(\cb, \cf) \to \text{Acc}_\omega(\ca, \ce)$) by the following equation: $$(f \boxtimes g)^*(s)= g^* \circ s \circ f.$$
\end{rem}

\begin{prop} \label{tensored}
Topoi has tensors over $\text{Acc}_\omega$.
\end{prop}
\begin{proof}
Putting together the content of \Cref{boxtimesdef}, \Cref{tensorscottconstructionwelldefined} and \Cref{tensor1cells}, we only need to show that $\boxtimes$ has the correct universal property, that is: $$\text{Topoi}(\ca \boxtimes \ce, \cf) \cong \text{Acc}_\omega(\ca, \text{Topoi}(\ce, \cf)). $$ When we spell out the actual meaning of the equation above, we discover that we did all the relevant job in the previous remarks. Indeed the biggest obstruction was the well-posedness of the definition.

\begin{align*} 
\text{Topoi}(\ca \boxtimes \ce, \cf) \cong &  \text{Cocontlex}(\cf, \ca \boxtimes \ce)  \\ 
 \cong & \text{Cocontlex}(\cf, \text{Acc}_\omega(\ca, \ce))  \\
\cong & \text{Cat}_{\text{cocontlex,}\text{acc}_\omega}(\cf \times \ca, \ce) \\
\cong &  \text{Acc}_\omega(\ca, \text{Cocontlex}(\cf,\ce))  \\
 \cong &  \text{Acc}_\omega(\ca, \text{Topoi}(\ce,\cf)).
\end{align*}
\end{proof}

\subsection{The Scott adjunction revisited} \label{ssrevisited}

\begin{rem}[Yet another proof of the Scott adjunction]
Let us start by mentioning that we can re-obtain the Scott adjunction directly from the fact that Topoi is tensored over $\text{Acc}_\omega$. Indeed if we evaluate the equation in \Cref{tensored} when $\ce$ is the terminal topos Set, $$\text{Topoi}(\ca \boxtimes \Set, \cf) \cong \text{Acc}_\omega(\ca, \text{Topoi}(\Set, \cf)) $$  we obtain precisely the statement of the Scott adjunction, $$\text{Topoi}(\mathsf{S}(\ca), \cf) \cong \text{Acc}_\omega(\ca, \mathsf{pt}(\cf)). $$ Being tensored over $\text{Acc}_\omega$ means in a way to have a relative version of the Scott adjunction.
\end{rem}

\begin{rem}\label{idiot}
Among natural numbers, we find extremely familiar the following formula, \[(30 \times 5) \times 6 = 30 \times (5 \times 6).\]
Yet, this formula yields an important property of the category of sets. Indeed $\Set$ is tensored over itself and the tensorial structure is given by the product. The formula above tells us that the tensorial structure of $\Set$ associates over its product.
\end{rem}

\begin{rem}[Associativity of $\boxtimes$ with respect to $\times$]
Recall that the category of topoi has products, but they are very far from being computed as in Cat. Pitts has shown \cite{pitts1985product} that $\ce \times \cf \cong \text{Cont}(\ce^\circ, \cf)$. This description, later rediscovered by Lurie, is crucial to get a slick proof of the statement below. 
\end{rem}

\begin{prop} Let $\ca$ be a finitely accessinble category. Then,
\[\ca \boxtimes (\ce \times \cf) \simeq ( \ca \boxtimes \ce ) \times \cf.\]
\end{prop}
\begin{proof}
We show it by direct computation.
\begin{align*} 
\ca \boxtimes (\ce \times \cf) \simeq &  \text{Acc}_\omega (\ca, \text{Cont}(\ce^\circ, \cf) )  \\ 
\simeq & \text{Cat}(\ca_\omega, \text{Cont}(\ce^\circ, \cf) ) \\
\simeq & \text{Cont}(\ce^\circ, \text{Cat}(\ca_\omega,\cf)) \\
\simeq & \text{Cont}(\ce^\circ, \text{Acc}_\omega (\ca,\cf)) \\
\simeq & ( \ca \boxtimes \cf ) \times \ce.
\end{align*}
\end{proof}

\begin{rem}
Similarly to \Cref{idiot}, the following display will appear completely trivial,
\[(30 \times 1) \times 6 = 30 \times (1 \times 6).\]
Yet, we can get inspiration from it, to unveil an important simplification of the tensor $\ca \boxtimes \ce$. We will show that it is enough to know the Scott topos $\mathsf{S}(\ca)$ to compute $\ca \boxtimes \ce$, at least when $\ca$ is finitely accessible.
\end{rem}

\begin{prop}[Interaction between $\boxtimes$ and Scott]
Let $\ca$ be a finitely accessinble category. Then,
$$\ca \boxtimes (-) \cong \mathsf{S}(\ca) \times (-).$$
\end{prop}
\begin{proof}
$$\ca \boxtimes (-) \cong \ca \boxtimes (\Set \times -) \cong (\ca \boxtimes \Set) \times (-).$$
\end{proof}

\begin{prop}[Powers and exponentiable Scott topoi] \label{powerspowers}
Let $\ca$ be a finitely accessible category. Then Topoi has powers with respect to $\ca$.
Moreover, $\ce^\ca$ is given by the exponential topos $\ce^{\mathsf{S}(\ca)}.$
\end{prop}
\begin{proof}
The universal property of the power object $\ce^\ca$ is expressed by the following equation,
$$\text{Topoi}(\cf, \ce^\ca) \cong \text{Acc}_\omega(\ca, \text{Topoi}(\cf, \ce)).$$ Now, because we have tensors, this is saying that $\text{Topoi}(\cf, \ce^\ca) \cong \text{Topoi}(\ca \boxtimes \cf, \ce)).$ Because of the previous proposition, we can  gather this observation in the following equation.
$$\text{Topoi}(\cf, \ce^\ca) \cong  \text{Topoi}(\mathsf{S}(\ca) \times \cf, \ce)).$$ This means that $\ce^\ca$ has the same universal property of the topos $\ce^{\mathsf{S}(\ca)}$ and thus exists if and only if the latter exists. By the well known characterization of exponentiable topoi, this happens if and only if $\mathsf{S}(\ca)$ is continuous which is of course true for presheaf categories.
\end{proof}

\begin{rem}
By inspecting the proof above, we see that Topoi has powers with respect to $\ca$ whenever $\mathsf{S}(\ca)$ is a continuous category. Of course, if $\ca$ is finitely accessible this condition is easily met. Yet, the proposition suggests a very natural series of questions. Can we find natural assumptions on $\ca$ such that $\mathsf{S}(\ca)$ is continuous or - less ambitiously - locally finitely presentable or - even less ambitiously - coherent (in the topos-theoretic sense)? 
\end{rem}

\section{Toolbox}\label{toolbox}

This section contains technical results on the Scott adjunction that will be extensively employed for more qualitative results. We study the behavior of $\mathsf{pt}$, $\mathsf{S}$ and $\eta$, trying to discern all their relevant properties. 
Before continuing, we briefly list the main results that we will prove in order to facilitate the consultation.
\begin{enumerate}
\item[\ref{embeddings}] $\mathsf{pt}$ transforms geometric embeddings in fully faithful functors.
\item[\ref{localicmaps}] $\mathsf{pt}$ transforms localic morphisms in faithful functors.
\item[\ref{surjections}] $\mathsf{S}$ maps pseudo-epis (of Cat) to geometric surjections.
\item[\ref{reflective}] $\mathsf{S}$ maps reflections to geometric embeddings.
\item[\ref{topological}] Introduces and studies the notion of topological embeddings between accessible categories.
\item[\ref{ff}] $\eta$ is faithful (and iso-full) if and only if $\ca$ has a faithful (and iso-full) functor into a finitely accessible category.
\end{enumerate}

\subsection{Embeddings \& surjections}

\begin{rem}
Observe that since $\mathsf{S}$ is a left adjoint, it preserves pseudo epimorphisms, analogously $\mathsf{pt}$ preserves pseudo monomorphisms. \Cref{embeddings} and \Cref{surjections} might be consequences of this observation, but we lack an explicit description of pseudo monomorphisms and pseudo epimorphisms in both categories. Notice that, instead, \Cref{reflective} represents a non-trivial behavior of $\mathsf{S}$.
\end{rem}

\subsubsection{On the behavior of $\mathsf{pt}$}
The functor $\mathsf{pt}$ behaves nicely with various notions of \textit{injective} or \textit{locally injective} geometric morphism.
\newline

Geometric embeddings between topoi are a key object in topos theory. Intuitively, they represent the proper notion of subtopos. It is a well known fact that subtopoi of a presheaf topos $\Set^C$ correspond to Grothendieck topologies on $C$ bijectively.

\begin{prop}  \label{embeddings}
 $\mathsf{pt}$ sends geometric embeddings to fully faithful functors.
\end{prop}
\begin{proof}
This is a relatively trivial consequence of the fact that the direct image functor is fully faithful but we shall include the proof in order to show a standard way of thinking. Let $i^*: \cg \leftrightarrows \ce: i_*$ be a geometric embedding. Recall, this means precisely that the $\ce$ is reflective in $\cg$ via this adjunction, i.e. the direct image is fully faithful. Let $p,q: \Set \rightrightarrows \ce$ be two points, or equivalently let $p^*, q^*: \ce \rightrightarrows \Set$ be two cocontinuous functors preserving finite limits. And let $\mu, \nu: p^* \Rightarrow q^*$ be two natural transformation between the points. 


The action of $\mathsf{pt}(i)$ on this data is the following. It maps $p^*$ to $p^*i^*$ while $\mathsf{pt}(i)(\mu)$ is defined by whiskering $\mu$ with $i^*$ as pictured by the diagram below.

\begin{center}
\begin{tikzcd}[row sep=2cm]
p^*\ar[d, bend left, Rightarrow, "\nu"]\ar[d, bend right, Rightarrow, "\mu"'] & p^*i^*\ar[d, bend left, Rightarrow, "\mu_{i^*}"]\ar[d, bend right, Rightarrow, "\nu_{i^*}"]\\
q^* & q^*i^*
\end{tikzcd}
\end{center}

Now, observe that $\mu \cong \mu_{i^*i_*}$ because $i^*i_*$ is isomorphic to the identity. This proves that $\mathsf{pt}(i)$ is faithful, in fact  $\mathsf{pt}(i)(\mu) =  \mathsf{pt}(i)(\nu)$ means that $\mu_{i^*} = \nu_{i^*}$, this implies that $\mu_{i^*i_*} = \nu_{i^*i_*}$, and so $\mu = \nu$. A similar argument shows that $\mathsf{pt}(i)$ is full (using that $i_*$ is full).
\end{proof}

Localic topoi are those topoi that appear as the category of sheaves over a locale. Those topoi have a clear topological meaning and represent a quite concrete notion of generalized space. Localic morphisms are used to generalize the notion of localic topos; a localic morphism $f: \cg \to \ce$ attests that there exist an internal locale $L$ in $\ce$ such that $\cg \simeq \mathsf{Sh}(L, \ce)$. In accordance with this observation, a topos $\cg$ is localic if and only if the essentially unique geometric morphism $ \cg \to \Set$ is localic.

\begin{defn} A morphism of topoi $f: \cg \to \ce$ is localic if every object in $\cg$ is a subquotient of an object in the inverse image of $f$.
\end{defn} 

\begin{prop}  \label{localicmaps}
$\mathsf{pt}$ sends localic geometric morphisms to faithful functors.
\end{prop}
\begin{proof}
Consider a localic geometric morphism $f: \cg \to \ce$. We shall prove that $\mathsf{pt}f$ is faithful on points. In order to do so,  let $p,q: \Set \rightrightarrows \ce$ be two points, or equivalently let $p^*, q^*: \ce \rightrightarrows \Set$ be two cocontinuous functors preserving finite limits. And let $\mu, \nu: p^* \Rightarrow q^*$ be two natural transformation between the points. 


We need to prove that if $\mu_{f^*} = \nu_{f^*}$, then $\mu = \nu$. In order to do so, let $e$ be an object in $\ce$. Since $f$ is a localic morphism, there is an object $g \in \cg$ and an epimorphism $l: f^*(g) \twoheadrightarrow e$\footnote{This is not quite true, we know that $e$ is a subquotient of $f^*(g)$, in the general case the proof gets a bit messier to follow, for this reason we will cover in detail just this case.}.

\begin{center}
\begin{tikzcd}[row sep=2cm]
p^*(e)\ar[d, bend left, Rightarrow, "\nu"]\ar[d, bend right, Rightarrow, "\mu"'] & p^*i^*(g)\ar[d, bend left, Rightarrow, "\mu_{i^*g}"]\ar[d, bend right, Rightarrow, "\nu_{i^*g}"] \ar[l, "p^*(l)"]\\
q^*(e) & q^*i^*(g) \ar[l, "q^*(l)"]
\end{tikzcd}
\end{center}

Now, we know that $\mu \circ p^*(l) = q^*(l) \circ \mu_{i^*g}$ and $\nu \circ p^*(l) = q^*(l) \circ \nu_{i^*g}$, because of the naturality of $\mu$ and $\nu$. Since  $\mu_{i^*} = \nu_{i^*}$, we get \[\mu \circ p^*(l) = q^*(l) \circ \mu_{i^*g} = q^*(l) \circ \nu_{i^*g} = \nu \circ p^*(l) .\]

Finally observe that $p^*(l)$ is an epi, because $p^*$ preserves epis, and thus we can cancel it, obtaining the thesis.

\end{proof}

\begin{prop} Let $f: \cg \to \ce$ be a geometric morphism. The following are equivalent.
\begin{itemize}
\item For every point $j: \Set \to \ce$ the pullback $\cg \times_{\ce} \Set$ has a point.
\item $\mathsf{pt}(f)$ is surjective on objects.
\end{itemize}
\end{prop}
\begin{proof}
Trivial.
\end{proof}

\subsubsection{On the behavior of $\mathsf{S}$}
The functor $\mathsf{S}$ behaves nicely with respect to epis, as expected. It does not behave nicely with any notion of monomorphism. In the next section we study those accessible functors $f$ such that $\mathsf{S}(f)$ is a geometric embedding.
\newline

\begin{prop}\label{surjections} $\mathsf{S}$ maps pseudo-epis (of Cat) to geometric surjections.
\end{prop}
\begin{proof}
See \cite[4.2]{laxepi}.
\end{proof}


%
%
%
%

\subsubsection{Topological embeddings}\label{topological}

\begin{defn}
Let $f: \ca \to \cb$ be a $1$-cell in $\text{Acc}_{\omega}$. We say that $f$ is a topological embedding if $\mathsf{S}(f)$ is a geometric embedding.
\end{defn}

This subsection is devoted to describe topological embeddings between accessible categories with directed colimits. The reader should expect such a quest to be highly nontrivial and rather technical, because $\mathsf{S}$ is a left adjoint and is not expected to have nice behavior on any kind of monomorphism. As a fact, we lack of a satisfying characterization of topological embeddings.

Fortunately we will manage to provide some useful partial results. Let us list the lemmas that we are going to prove.
\begin{enumerate}
    \item[\ref{necessary}] a necessary condition for a functor to admit a topological embedding into a finitely accessible category
    \item[\ref{sufficient}] a sufficient and quite easy to check criterion for a functor to be a topological embedding
    \item[\ref{intofinitely}]  a full description of topological embeddings into finitely accessible categories

\end{enumerate}

\begin{rem}
Topological embeddings into finitely accessible categories $i: \ca \to \mathsf{Ind}(C)$ are very important because $\mathsf{S}(i)$ will describe, by definition, a subtopos of $\Set^C$. This means that there exist a topology $J$ on $C$ such that $\mathsf{S}(\ca)$ is equivalent to $\mathsf{Sh}(C,J)$, this leads to concrete presentations of the Scott topos.
\end{rem}

\begin{lem}[A necessary condition]\label{necessary}
If $\ca$ has a fully faithful topological embedding $f: \ca \to \mathsf{Ind}(C)$ into a finitely accessible category, then $\eta_A : \ca \to \mathsf{ptS}(\ca)$ is fully faithful.
\end{lem}
\begin{proof}
Assume that $\ca$ has a topological embedding $f: \ca \to \mathsf{Ind}(C)$ into a finitely accessible category. This means that $\mathsf{S}(f)$ is a geometric embedding. Now, we look at the following diagram.
    \begin{center}
    \begin{tikzcd}
\ca \arrow[rr, "f" description] \arrow[dd, "\eta_\ca"]              &  & \mathsf{Ind}(C) \arrow[dd, "\eta_{\mathsf{Ind}(C)}" description, two heads, hook] \\
                                                                    &  &                                                                                   \\
\mathsf{pt} \mathsf{S}(\ca) \arrow[rr, "\mathsf{ptS}f" description] &  & \mathsf{pt} \mathsf{S}(\mathsf{Ind}(C))                                          
\end{tikzcd}
    \end{center}
\Cref{trivial} implies that $\eta_{\mathsf{Ind}(C)}$ is an equivalence of categories, while \Cref{embeddings} implies that $\mathsf{ptS}(f)$ is fully faithful. Since also $f$ is fully faithful, $\eta_\ca$ is forced to be fully faithful.
\end{proof}


\begin{thm}\label{reflective}\label{sufficient}
Let $i: \ca \to \cb$ be a $1$-cell in $\text{Acc}_\omega$ exhibiting $\ca$ as a reflective subcategory of $\cb$ \[L: \cb \leftrightarrows \ca: i. \] Then $i$ is a topological embedding.
\end{thm}


 \begin{proof}
 We want to show that $\mathsf{S}(i)$ is a geometric embedding. This is equivalent to show that the counit $i^*i_*(-) \Rightarrow (-)$ is an isomorphism. Going back to \Cref{f_*}, we write down the obvious computations, \[i^*i_*(-) \cong (i^* \circ r_\cb \circ \ran_i \circ \iota_\ca) (-).\] Now, observe that since $i$ has a left adjoint $L$ the operator $\ran_i$ just coincides with $(-) \circ L$, thus we can elaborate the previous equation as follows.
 \[ (i^* \circ r_\cb \circ \ran_i \circ  \iota_\ca)  (-) \cong (i^*\circ r_\cb)(- \circ L),\]
 Now, $(- \circ L)$ will preserve directed colimits because is the composition of a cocontinuous functor with a functor preserving directed colimits. This means that it is a fixed point of $r_\cb$.
 \[ (i^*\circ r_\cb)((-) \circ L) \cong i^*((-) \circ L) \cong  (-) \circ L \circ i \cong (-).\] The latter isomorphism is just the definition of reflective subcategory. This concludes the proof.
 \end{proof}

\begin{thm}\label{intofinitely}
$f: \ca \to \mathsf{Ind}(C)$ is a topological embedding into a finitely accessible category if and only if, for all $p: \ca \to \Set$ preserving directed colimits, the following equation holds (whenever well defined), \[\lan_i(\ran_f(p) \circ i) \circ f \cong p.\]
\end{thm}

\begin{proof}
The result follows from the discussion below.
\end{proof}

\begin{rem}[$f_*$ and finitely accessible categories]

Given a $1$-cell $f: \ca \to \cb$ in $\text{Acc}_{\omega}$, we experienced that it can be quite painful to give an explicit formula for the direct image functor $f_*$. In this remark we improve the formula provided in \Cref{f_*} in the special case that the codomain if finitely accessible. In order to do so we study the diagram of \Cref{f_*}. To settle the notation, call $f: \ca \to \mathsf{Ind}(C)$ our object of study and $i: C \to \mathsf{Ind}(C)$ the obvious inclusion.

\begin{center}
\begin{tikzcd}
\mathsf{S}\ca \arrow[ddd, "\iota_\ca"] \arrow[rr, "f_*" description, bend right] &  & \mathsf{S}\mathsf{Ind}(C) \arrow[ddd, "\iota_{\mathsf{Ind}(C)}"] \arrow[ll, "f^*"'] \\
 &  &  \\
 &  &  \\
{\P(\ca)} \arrow[rr, "\ran_f" description, bend right] \arrow[rr, "\lan_f" description, bend left] &  & {\P(\mathsf{Ind}(C))} \arrow[ll, "f^*" description]
\end{tikzcd}
\end{center}

We are use to this diagram from \Cref{f_*}, where we learnt also the following formula \[f_* \cong r_\cb \circ \ran_f \circ \iota_\ca.\] We now use the following diagram to give a better description of the previous equation.

\begin{center}
\begin{tikzcd}
\mathsf{S}\ca \arrow[ddd, "\iota_\ca"] \arrow[rr, "f_*" description, bend right]                       &  & \mathsf{S}\mathsf{Ind}(C) \arrow[ddd, "\iota_{\mathsf{Ind}(C)}"] \arrow[ll, "f^*"'] \arrow[rr, "i^*" description, two heads, tail, bend right] &  & \Set^C \arrow[ll, "\lan_i" description, two heads, tail] \arrow[llddd, "\lan_i" description] \\
                                                                                                       &  &                                                                                                                                                &  &                                                                                              \\
                                                                                                       &  &                                                                                                                                                &  &                                                                                              \\
{\P(\ca)} \arrow[rr, "\ran_f" description, bend right] \arrow[rr, "\lan_f" description, bend left] &  & {\P(\mathsf{Ind}(C))} \arrow[ll, "f^*" description] \arrow[rruuu, "i^*" description, bend right]                                           &  &                                                                                             
\end{tikzcd}
\end{center}

We claim that in the notations of the diagram above, we can describe the direct image $f_*$ by the following formula, \[f_* \cong \lan_i \circ  i^* \circ \ran_f \circ \iota_\ca,\]

this follows from the observation that $r_{\mathsf{Ind}(C)}$ coincides with $\lan_i \circ  i^*$ in the diagram about.
\end{rem}

\subsection{A study of the unit $\eta_\ca$} 

This section is devoted to a focus on the unit of the Scott adjunction. We will show that good properties of $\eta_\ca$ are related to the existence of \textit{finitely accessible representations} of $\ca$.
A weaker version of the following proposition appeared in \cite[2.6]{simon}. Here we give a different proof, that we find more elegant and provide a stronger statement. These kind of results were the original motivation for the whole theory behind the Scott adjunction, as it is in general very hard to tell whether an accessible category has a finitely accessible representation. More recently, \cite{lieberman2019hilbert} has followed the techniques introduced by \cite{simon} to takle a similar problem.

\begin{prop}\label{ff} The following are equivalent:
\begin{enumerate}
    \item  The unit at $\ca$ of the Scott adjunction $\ca \to \mathsf{pt}  \mathsf{S} \ca$ is faithful (and iso-full);
    \item  $\ca$ admits a faithful (and iso-full) functor $f: \ca \to \mathsf{Ind}(C)$ preserving directed colimits;
\end{enumerate}
\end{prop}
\begin{proof}
\begin{itemize}
    \item[$1) \Rightarrow 2)$] Assume that $\eta_\ca$ is faithful. Recall that any topos admits a geometric embedding in a presheaf category, this is true in particular for $\mathsf{S}(\ca)$. Let us call $\iota$ some such geometric embedding $\iota: \mathsf{S}(\ca) \to \Set^X$. Following \Cref{embeddings} and \Cref{trivial}, $\mathsf{pt}(\iota)$ is a fully faithful functor into a finitely accessible category $\mathsf{pt}(\iota): \mathsf{ptS}(\ca) \to \mathsf{Ind}(X)$. Thus the composition $\mathsf{pt}(\iota) \circ \eta_A$ is a faithful functor into a finitely accessible category \[\ca \to \mathsf{pt}  \mathsf{S} \ca \to \mathsf{Ind}(X).\] Obverse that if $\eta_\ca$ is iso-full, so is the composition $\mathsf{pt}(\iota) \circ \eta_A$.
    \item[$2) \Rightarrow 1)$] Assume that $\ca$ admits a faithful functor $f: \ca \to \mathsf{Ind}(C)$ preserving directed colimits. Now we apply the monad $\mathsf{pt}\mathsf{S}$ obtaining the following diagram.
    \begin{center}
    \begin{tikzcd}
\ca \arrow[rr, "f" description] \arrow[dd, "\eta_\ca"]              &  & \mathsf{Ind}(C) \arrow[dd, "\eta_{\mathsf{Ind}(C)}" description, two heads, hook] \\
                                                                    &  &                                                                                   \\
\mathsf{pt} \mathsf{S}(\ca) \arrow[rr, "\mathsf{ptS}f" description] &  & \mathsf{pt} \mathsf{S}(\mathsf{Ind}(C))                                          
\end{tikzcd}
    \end{center}
  \Cref{trivial} implies that $\eta_{\mathsf{Ind}(C)}$ is an equivalence of categories, thus $\mathsf{ptS}(f) \circ \eta_\ca$ is (essentially) a factorization of $f$. In particular, if $f$ is faithful, so has to be $\eta_\ca$. Moreover, if $f$ is iso-full and faithful, so must be $\eta_\ca$, because this characterizes pseudo-monomorphisms in Cat (and by direct verification also in $\text{Acc}_\omega$).
\end{itemize}
\end{proof}

 \begin{rem}
 If we remove iso-fullness from the statement we can reduce the range of $f$ from any finitely accessible category to the category of sets.

\begin{prop}\label{fff} The following are equivalent:
\begin{enumerate}
    \item $\ca$ admits a faithful functor $f: \ca \to \Set$ preserving directed colimits.
    \item  $\ca$ admits a faithful functor $f: \ca \to \mathsf{Ind}(C)$ preserving directed colimits;
\end{enumerate}
\end{prop}
\begin{proof}
The proof is very simple. $1) \Rightarrow 2)$ is completely evident. In order to prove $2) \Rightarrow 1)$, observe that since $\mathsf{Ind}(C)$ is finitely accessible, there is a faithful functor $ \cy :  \mathsf{Ind}(C) \to \Set$ preserving directed colimits given by \[ \cy := \coprod_{p \in C} \mathsf{Ind}(C)(p, -). \]    The composition $g:= \cy \circ f$ is the desired functor into $\Set$.
\end{proof}
\end{rem}

\newpage
\section*{Acknowledgements}
I am indebted to my advisor, \textit{Jiří Rosický}, for the freedom and the trust he blessed me with during these years, not to mention his sharp and remarkably blunt wisdom. I would like to thank \textit{Simon Henry}, for his collaboration in those days in which this project was nothing but an informal conversation at the whiteboard. I am grateful to \textit{Peter Arndt}, for his sincere interest in my research, and the hint of looking at the example of the geometric theory of fields. Finally, I am grateful to the \textit{anonymous referee} for their suggestions and comments.

\section*{Data availability}
Data sharing not applicable to this article as no datasets were generated or analysed during the current study.

\bibliography{thebib}
\bibliographystyle{alpha}

\end{document}